\documentclass[oneside,12pt]{article}  

\usepackage{graphicx}
\usepackage{amsmath}
\usepackage{amsfonts}
\usepackage{amsthm}

\usepackage{enumerate}
\usepackage{doc}
\usepackage{ifthen}
\usepackage{nextpage}

\newtheorem{theorem}{Theorem}
\newtheorem{lemma}{Lemma}
\newtheorem{definition}{Definition}  
\newtheorem{corollary}{Corollary}
\newtheorem{proposition}{Proposition}
\newtheorem{claim}{Claim}

\begin{document}

\begin{center}
{\Large {\bf On a Problem in Diophantine Approximation}}\\
\vspace{2mm}
{\bf Evgeni Dimitrov, Yakov Sinai }
\end{center}

\vspace{2mm}

\begin{center}
{\bf Abstract}\\
\end{center}
We prove new results, related to the Littlewood and Mixed Littlewood conjectures in Diophantine approximation.

\section{Introduction}

A classical result by Dirichlet states that, for any real number $\alpha$ there exist infinitely many positive integers $q$ such that 
\begin{equation}\label{Dirichlet}
q \cdot ||q\alpha|| < 1,
\end{equation}
where $|| \cdot ||$ denotes the distance to the nearest integer. A consequence of Hurwitz's theorem is that the right-hand side of (\ref{Dirichlet}) cannot be improved by an arbitrary positive constant. Specifically, for $\epsilon < 1/\sqrt{5}$  there exist real numbers $\alpha$, for which the inequality $||q\alpha|| < \epsilon q^{-1}$ has at most finitely many solutions. The numbers $\alpha$ are called {\em badly approximable numbers} and we denote the set of all such numbers by {\bf Bad}; that is,
$$ {\bf Bad} = \{ \alpha \in \mathbb{R} : \exists c(\alpha) > 0 \mbox{ so that } ||q\alpha|| > c(\alpha)q^{-1} \mbox{ for all } q \in \mathbb{N} \}.$$
It is well-known that a real number lies in {\bf Bad} if, and only if, it has bounded partial coefficients in its continued fraction expansion. We also recall that {\bf Bad} has Lebesgue measure zero and full Hausdorff dimension \cite{Jarnik28}.\\

Dirichlet's result coupled with the elementary fact that $||x|| < 1$ for any real $x$, implies that for any given pair $(\alpha, \beta)$ there exist infinitely many positive integers $q$, which satisfy
$$ q \cdot ||q\alpha|| \cdot ||q\beta|| < 1.$$
A famous open problem in simultaneous Diophantine approximation, called Littlewood's conjecture \cite{Littlewood68}, asserts that in fact for any given pair $(\alpha, \beta)$ of real numbers the following stronger result holds
\begin{equation}\label{Littlewood}
\inf_{q\geq 1} q \cdot ||q\alpha|| \cdot ||q\beta || = 0.
\end{equation}
A direct application of (\ref{Dirichlet}) shows that (\ref{Littlewood}) is satisfied whenever $1$, $\alpha$ and $\beta$ are linearly dependent over $\mathbb{Q}$. Moreover the conjecture is trivially satisfied whenever $\alpha \not\in {\bf Bad}$ or $\beta \not\in {\bf Bad}$. Consequently, the only case of interest for the conjecture is when $\alpha$ and $\beta$ are both badly approximable.\\

There have been many contributions towards resolving Littlewood's conjecture, of which we will presently mention only two. In \cite{Velani00} Pollington and Velani established for each $\alpha \in {\bf Bad}$ the existence of a thick subset $G(\alpha) \subset {\bf Bad}$, such that $(\alpha, \beta)$ satisfy $(\ref{Littlewood})$ for each $\beta \in G(\alpha)$. The proof of this results rests on the combination of harmonic analysis and tools from metric number theory, as well as the use of the {\em Kaufman measure}, constructed in \cite{Kaufman80}. More recently, Einsiedler, Katok and Lindenstrauss, proved the remarkable result that the set of pairs $(\alpha, \beta)$, for which (\ref{Littlewood}) does not hold has Hausdorff dimension zero  \cite{EKL06}. The authors actually established a partial result toward a conjecture by Margulis on ergodic actions on the homogeneous space $SL_k(\mathbb{R})/SL_k(\mathbb{Z})$ for $k \geq 3$. See \cite{Venkatesh08} for a more informal discussion of the work of Einsiedler, Katok and Lindenstrauss and \cite{Moshchevitin12} for a survey on Littlewood's conjecture and related problems.\\

A simple, yet natural, approach to consider for proving Littlewood's conjecture is to restrict one's attention to subsequences of the positive integers formed by the denominators of the $n$th convergents of $\alpha$ or $\beta$. In particular, let $\alpha = [a_0; a_1, a_2, a_3, \cdots]$ represent the regular continued fraction expansion of $\alpha$, and $p_n(\alpha)/q_n(\alpha) = [a_0; a_1,\cdots,a_n]$ denote its $n$th convergent. Then using the fact that the convergents are the best approximates (i.e. $||q\alpha|| \geq ||q_n(\alpha)\alpha||$ for any $q \leq q_n(\alpha)$) one readily has that $q_n(\alpha)||q_n(\alpha) \alpha|| < 1$. Then intuitively one would hope that the quantity $||q_n(\alpha)\beta||$ would be small for some $n$, so that $(\alpha, \beta)$ satisfy (\ref{Littlewood}). Unfortunately the latter idea proves to be insufficient as shown in Theorem 2 in \cite{Velani00}:
\begin{theorem}\label{PV2}
Given $\alpha \in {\bf Bad}$ and $\lambda \in (0,1)$, there exists a subset ${\bf B}_{\lambda}(\alpha)$ of {\bf Bad} with $\dim_H {\bf B}_{\lambda}(\alpha) = \lambda$, such that for any $\beta \in {\bf B}_{\lambda}$,
$$||q_n(\alpha) \beta|| \geq \delta \mbox{ for all } n\in\mathbb{N},$$
where $\delta = \delta(\alpha, \lambda) > 0$ is a constant, and $\dim_H$ denotes the Hausdorff dimension. 
\end{theorem}
The above result is connected with a question posed by P. Erd\H{o}s in \cite{Erdos75}:
Given a sequence of integers $n_1 < n_2 < n_3 < \cdots$, satisfying $n_{k+1}/n_k \geq \lambda > 1$ for $k = 1, 2, \cdots$, does there exist a $\xi$ for which the sequence $\{n_k \xi\}$ is not everywhere dense? ($\{ x \}$ denotes the fractional part of $x$).

A sequence $\{n_k\}$ of the type above is called a {\em lacunary sequence}. The above question has been answered independently by A. Pollington \cite{Pollington79} and B. de Mathan \cite{deMathan80} who showed that the set of $\xi$ as above has full Hausdorff dimension for any lacunary sequence $\{n_k\}$. The connection of these results with Theorem \ref{PV2} is made once one uses the well-known fact (see e.g. \cite{Khinchin64}):
\begin{equation}\label{lac}
 \frac{q_n(\alpha)}{q_{n-1}(\alpha)} = a_n + [a_{n-1},\cdots,a_1] \mbox{ for } n\geq 1.
\end{equation}
Then if $\alpha \in {\bf Bad}$ the above expression is bounded from below by a universal constant $K(\alpha) > 1$, for all $n$, which implies $\{q_n(\alpha)\}$ is a lacunary sequence.\\

As discussed by Pollington and Velani, Theorem \ref{PV2} shows that there is ``absolutely no hope of proving Littlewood's conjecture by simply looking at the convergents." The result of the Theorem \ref{PV2} indicates that one cannot prove the conjecture if only the convergents of one of the two numbers is considerd; however it does not rule out the possibility of proving (\ref{Littlewood}) by looking at the convergents of both $\alpha$ and $\beta$. We formulate two specific questions along those lines below:\\

\setlength{\parindent}{0pt}

{\em Question 1.} Do there exist $(\alpha, \beta) \in {\bf Bad^2}$ such that $\inf_{n\geq1}q_n(\alpha)||q_n(\alpha) \alpha||||q_n(\alpha) \beta|| > 0$ and $\inf_{n\geq1}q_n(\beta)||q_n(\beta) \alpha||||q_n(\beta) \beta|| > 0$.\\

\vspace{2mm}

{\em Question 2.} If $\alpha \in {\bf Bad}$, what is the size of $A(\alpha) \subset {\bf Bad}$,  such that for $\beta \in A(\alpha)$ one has $\inf_{n\geq1}q_n(\alpha)||q_n(\alpha) \alpha||||q_n(\alpha) \beta|| > 0$ and $\inf_{n\geq1}q_n(\beta)||q_n(\beta) \alpha||||q_n(\beta) \beta|| > 0$.\\

\vspace{2mm}

\setlength{\parindent}{20pt}

The purpose of this paper is to establish partial answers to the above two questions, with our main results presented in the next section. In the questions above we have specifically restricted our attention to the case when both $\alpha$ and $\beta$ are badly approximable numbers, since that is the most relevant case for Littlewood's conjecture. In addition, in view of the fact that the Lebesgue measure of {\bf Bad} is $0$, the set $A(\alpha)$ is also of measure zero, hence our discussion of the size of $A(\alpha)$, will be in terms of its Hausdorff dimension. Finally, we remark that it is somewhat surprising that the above questions have to date escaped attention and, to the author's knowledge, this paper is an initial attempt to address them.

\section{Main Results}

Our first result shows that there are many pairs $(\alpha, \beta)$ of real numbers, for which Littlewood's conjecture fails, if one restricts it to the denominators of the convergents of $\alpha$ and $\beta$.

\begin{proposition}\label{prop1}
Let 
$$A = \{ (x,y) \in {\bf Bad}^2 \ | \exists c(x,y) > 0 \mbox{ s.t. }||q_n(x)x||||q_n(x)y||q_n(x) \geq c(x,y) \mbox{ and}$$
$$||q_n(y)x||||q_n(y)y||q_n(y) \geq c(x,y)\}.$$
Then 
$$\frac{3}{2} \leq \dim_HA \leq 2.$$
\end{proposition}

Since $A \subset \mathbb{R}^2$, we trivially obtain the bound $\dim_HA \leq 2$. The lower bound of the dimension in the proposition will be the focus of our discussion. The proof of Proposition \ref{prop1} will be carried out by a reduction to one dimension, and application of Proposition \ref{prop2}, which we state below.

\begin{proposition}\label{prop2}
Let $x \in {\bf Bad}$ and let $A$ be as in Proposition \ref{prop1}. Denote by $A(x) = A \cap L_x$, where $L_x$ is the vertical line in $\mathbb{R}^2$, passing through $x$. Then $\dim_H A(x) \geq \frac{1}{2}.$
\end{proposition}

The reduction from Proposition \ref{prop1} to Proposition \ref{prop2} is established using the following corollary.
\begin{corollary}\label{Fal712}
Let $F$ be any subset of $\mathbb{R}^2$ and let $E$ be a subset of the $x$-axis. If $\dim_H (F \cap L_x) \geq t$ for all $x \in E$, then $\dim_HF \geq t + \dim_HE$. ($L_x$ is the line parallel to the $y$-axis passing through the point $x$).
\end{corollary}
\begin{proof}
This is Corollary 7.12, pp.106 in \cite{Falconer}.
\end{proof}
Assuming the validity of Proposition \ref{prop2}, we have by Corollary \ref{Fal712} that $\dim_H A \geq \frac{1}{2} + \dim_H{\bf Bad}$, and since $\dim_H {\bf Bad} = 1$, the conclusion of Proposition \ref{prop1} follows. Hence our problem reduces to proving Proposition \ref{prop2}.

\section{Auxiliary Results}

Before we go into the proof of Proposition \ref{prop2} we fix some notation and present some results to be used throughout the rest of the paper.\\

To ease our notation we write $[a,b]$ for the interval $[a,b]$ if $a < b$ and $[b,a]$ if $a > b$.
\begin{proposition}
Let $\{a_k\}_{k\geq 0}$ be a sequence of nonnegative integers with $a_i > 0$ for $i \geq 1$. Denote by $\frac{p_n}{q_n} = [a_0; a_1,...,a_n]$ and $\frac{p_{n-1}}{q_{n-1}} = [a_0; a_1,...,a_{n-1}]$. Then the set of numbers $X_n = \{x \in \mathbb{R}/\mathbb{Q} | x = [b_0; b_1,...,b_n,...] $ and $b_i = a_i$ for $0 \leq i \leq n \}$ is the set $I_n = \big[\frac{p_n}{q_n}, \frac{p_{n-1} + p_n}{q_{n-1} + q_n}\big ]/\mathbb{Q}$.  
\end{proposition}
\begin{proof}
We proceed by indutction on $n$.\\
Base case: n = 0.
As is typical in the theory of continued fractions we use the convention that $q_{-1} = 0$ and $p_{-1} = 1$. Then we have that $x \in X_0 \iff b_0 = a_0 \iff  a_0 \leq x < a_0 + 1\iff x \in I_0$.

So now suppose that we have proved the statement for $n \leq k$, and consider $n = k + 1$. Let $x \in X_n$ and $x = [a_0; a_1,...,a_n,...]$.  We set $r_1 = [a_1;a_2,...,a_n,...]$ and we have that $x = a_0 + 1/r_1$. Moreover if $p'_s/q'_s$, denotes the $s$-th order convergent of $r_1$ we have that (see e.g. \cite{Khinchin64} pp.4)
$$p_s = a_0p'_{s-1} + q'_{s-1}$$
$$q_s = p'_{s-1}.$$
Consequently we apply our induction hypothesis to $r_1$, and obtain that it lies in the set $I'_k = \big[\frac{p'_k}{q'_k}, \frac{p'_{k-1} + p'_k}{q'_{k-1} + q'_k}\big ]/\mathbb{Q}$. Hence $a_0 + 1/r_1 \in \big[a_0 + \frac{q'_k}{p'_k}, a_0 + \frac{q'_{k-1} + q'_k}{p'_{k-1} + p'_k}\big]$.
$$a_0 + \frac{q'_k}{p'_k} = a_0 + \frac{p_n - a_0q_n}{q_n} = \frac{p_n}{q_n}$$
$$a_0 + \frac{q'_{k-1} + q'_k}{p'_{k-1} + p'_k} = a_0 + \frac{p_{n-1} - a_0q_{n-1} + p_n - a_0q_n}{q_{n-1} + q_n} = \frac{p_n + p_{n-1}}{q_n + q_{n-1}}$$
Hence we obtain $x\in I_n$.\\

Conversely suppose that $x \in I_n$. In particular, we have that $x \in I_0$ and thus by induction hypothesis $x = [a_0; b_1,...,b_n,...]$. As before we set $r_1 = [b_1;b_2,...,b_n,...]$, so that $x = a_0 + 1/r_1$, and let $p'_s/q'_s$, denote the $s$-th order convergent of $[a_1;a_2,...,a_n,...]$.
We have that since $x\in I_n$ then $1/r_1 \in \big[\frac{p_n}{q_n} - a_0, \frac{p_{n-1} + p_n}{q_{n-1} + q_n} - a_0\big ]$ and
$$\big[\frac{p_n}{q_n} - a_0, \frac{p_{n-1} + p_n}{q_{n-1} + q_n} - a_0\big ] =  \big[\frac{p_n - a_0q_n}{q_n}, \frac{p_{n-1} + p_k - a_0q_{n} - a_0q_{n - 1}}{q_{n} + q_{n-1}} \big ] = \big[\frac{q'_{k}}{p'_{k}}, \frac{q'_{k} + q'_{k-1}}{p'_{k} + p'_{k-1}} \big].$$
Hence $r_1 \in \big[\frac{p'_{k}}{q'_{k}}, \frac{p'_{k} + p'_{k-1}}{q'_{k} + q'_{k-1}} \big]$. Applying the induction hypothesis we get that $r_1 = [a_1;a_2,...,a_n,...]$, which proves that $x \in X_n$.

\end{proof}

We shall call an interval of the above form a {\em fundamental interval of order $n$}, and specify it uniquely as $I(a_0, ..., a_n) =  \big[\frac{p_n}{q_n}, \frac{p_{n-1} + p_n}{q_{n-1} + q_n}\big ]$. Whenever we work in the unit interval, where $a_0 = 0$, we shall omit it from the notation.\\

{\em Some remarks:} It is easy to see that two fundamental intervals: $I(a_1,...,a_n)$, $I(b_1,...,b_n)$ of the same order are disjoint except possibly at their endpoints. Indeed, if the interiors of the two fundamental intervals intersect non-trivially we will have an irrational point $x$, lying in both. Consequently, we will have that the continued fraction expansion of $x$ begins with $[a_1,...,a_n,...]$ and $[b_1,...,b_n,...]$. Given that the continued fraction expansion is unique we obtain $a_i = b_i$ for $1 \leq i \leq n$, implying  $I(a_1,...,a_n) = I(b_1,...,b_n)$. 

\vspace{2mm}

Since all irrationals in a fundamental interval of order $n$ agree on the first $n$ elements of their continued fraction, we have that they have the same convergents $p_s/q_s$, for $s \leq n$. We will call the numbers $q_1, \cdots, q_n$, the {\em associated denominators} of the fundamental interval $I(a_1,\cdots, a_n)$.

\vspace{2mm}

The following is a technical lemma that we will use in the proof of Proposition \ref{prop2}.

\begin{lemma}\label{spec}
Let $x \in \mathbb{R}$ be irrational and c a positive constant less than $1/4$. If $A,B$ be arbitrary positive integers such that $||Ax|| > c$, then at least one of the following inequalities holds:
$$||(A + B)x|| > c \hspace{5mm} ||(2A + B)x|| > c \hspace{5mm} ||(3A + B)x|| > c.$$
\end{lemma}
\begin{proof}
Supposing the contrary we have $||(A + B)x|| \leq c, ||(2A + B)x|| \leq c,  ||(3A + B)x|| \leq c.$ Let $a = \{Ax\}$ and $b = \{(A + B)x\}$. The assumptions in the lemma ensure that $c < a < 1 - c$ and $0 < b < 1$.
We consider four cases:

Case 1: $a < 1/2$ and  $b < 1/2$.\\
We have that $||(A + B)x|| = b$ and $0 < b < c$. Consequently $\{(2A + B)x\} = a + b$ and  $c < a + b < 1/2 + c$. If $1/2 \leq a + b < 1/2 + c$ we get $||(2A + B)x|| = 1- \{(2A + B)x\} \geq 1/2 - c > c$ (since $c < 1/4$), leading to a contradiction. If $1/2 \geq a + b > c$ we get $||(2A + B)x|| = \{(2A + B)x\}$ and $||(2A + B)x|| > c$, leading to a contradiction. \\

Case 2: $a > 1/2$ and  $b > 1/2$.\\
We have that $||(A + B)x|| = 1 -  b$ and $0 <1 - b \leq c \iff 1 - c \leq b < 1$. Consequently  $\{(2A + B)x\} = a + b - 1$ and $2 - c > a + b > 1/2 + (1 - c) > 1 + c $ (since $c < 1/4$). Then we get that $1-c > \{(2A + B)x\} > c$ or $||(2A + B)x|| > c$, which is a contradiction.\\

Case 3: $a > 1/2$ and  $b < 1/2$.\\
We have that $||(A + B)x|| = b$ and $0 < b < c$. Since $b < c$ and $1/2 < a < 1 - c$, we know that $\{(2A + B)x\} = a + b$ and thus $c \geq ||(2A + B)x|| = 1 - \{(2A + B)x\} =1 - a - b > 0$. Consequently $\{(3A + B)x\} = 2a + b -1$ and from $a >  2a + b - 1 \geq a - c$ we get $1- c > a > \{(3A + B)x\} \geq a - c > 1/2 -c > c$ . This implies $||(3A + B)x|| > c$ - contradiction.\\

Case 4: $a < 1/2$ and  $b > 1/2$.\\
We have that $||(A + B)x|| = 1 -  b$ and $1 - c \leq b < 1$. Consequently as $a > c$ we know $c \geq ||(2A + B)x|| = \{(2A + B)x\} = a + b - 1 > 0$ and $\{(3A + B)x\} = 2a + b - 1 > a$. Then we have $1 - c > 1/2 + c> a + c \geq 2a + b - 1 = \{(3A + B)x\} > a > c$, hence $||(3A + B)x|| > c$.\\

Since we reached a contradiction in all of the above cases, which are mutually exclusive and exhaustive we conclude the statement of the Lemma.

\end{proof}

\section{Proof of Proposition \ref{prop2}}

We consider the following construction. Let $[0,1] = E_0 \supset E_1 \supset E_2...$ be a decreasing sequence of sets with $E_k$ a union of a finite number of disjoint closed intervals (called {\em $k$-th level basic intervals}), with each interval of $E_k$ containing at least two intervals of $E_{k+1}$ and the maximum length of $k$-th level intervals tending to $0$ as $k\rightarrow \infty$. In addition, we define $F$ as
$$F = \cap_{n =0}^{\infty} E_n.$$

\begin{lemma}\label{Fal46}
With the construction described above suppose in addition that each $(k-1)$-th level interval contains at least $m_k\geq2$ $k$-th leavel intervals $(k = 1,2,...)$, which are seperated by gaps of size at least $\epsilon_k$, where $0 < \epsilon_{k+1} < \epsilon_k$. Then
$$ \dim_H F \geq \liminf \frac{\log(m_1 \cdots m_{k-1})}{-\log(m_k\epsilon_k)}.$$
\end{lemma}
\begin{proof}
This is Example 4.6, pp.64 in \cite{Falconer}.
\end{proof}

{\bf Step 1.} {\em An informal description of apporach.}\\
Before we go to the actual proof of Proposition \ref{prop2} we give an overview of our approach. Our aim is to construct collections $E_k$ of disjoint closed intervals in the notation of Lemma \ref{Fal46}, which will be appropriate subsets of the $k$-th order fundamental intervals, whose intersection is contained in $A(x)$. We will desire three properties from our chosen intervals:

\begin{enumerate}
\item The first $k$ elements of the continued fraction expansion of the numbers in an interval of $E_k$ (which are all the same by our remark in the previous section), are all bounded by a previously chosen constant $2M$, independent of $k$.
\item The first $k$ denominators of convergents of the numbers in any interval of $E_k$ (which are all the same by our remark in the previous section), all satisfy the condition $||q_nx|| > c$, for $n \leq k$. Here $c$ is some positive constant independent of $k$ and $n$.
\item For each number $y \in E_k$, we have that $||q_n(x)y|| > c$, for all $n$ satisfying $c^{1/2}M^{2(k-1)} \leq q_n(x) < c^{1/2}M^{2k}$. 
\end{enumerate}
The first condition, ensures that all the numbers in $\cap_k E_k$ are badly approximable. The second condition ensures that if $y \in \cap_k E_k$, then $||q_n(y)x|| > c$ for all $n \in \mathbb{N}$. The third condition ensures that if $y \in \cap_k E_k$, then $||q_n(x)y|| > c$ for all $n \in \mathbb{N}$. 

Finally, since $x,y$ are badly approximable we have that $q_n(x)||q_n(x)x|| > c(x)$ and $q_n(y)||q_n(y)y|| > c(y)$, for some constants $c(x), c(y) > 0$, which implies that $q_n(x)||q_n(x)x||||q_n(x)y|| > c(x)\cdot c$ and $q_n(y)||q_n(y)x||||q_n(y)y|| > c(y)\cdot c$. In particular, setting $c(x,y) = c \cdot \min\{c(x), c(y)\}$, we obtain that $(x,y) \in A(x)$, for all $y \in \cap_k E_k$. Consequently, any lower bound on the dimension of $\cap_k E_k$ will naturally hold for $A(x)$.\\

{\bf Step 2.} {\em Fixing the constants $c$ and $M$} \\
Fix $x \in {\bf Bad}$. We know (recall equation (\ref{lac})) that there exists a constant $\lambda(x) > 1$, such that $q_n(x)/q_{n-1}(x) \geq \\lambda$ for $n \geq 1$. In the following discussion all constants will depend on $x$, so we will omit it from our notation. \\

Let $M$ satisfy the following conditions:
\begin{equation}\label{M}
\begin{aligned}
&1.& M/32 >& 2\log_\alpha(M) + 1. \\
&2.& M >& 128.
\end{aligned}
\end{equation}

Next let $c(M)$ satisfy the following conditions:
\begin{equation}\label{c}
\begin{aligned}
&1.& &3c^{1/2}M^4 < 1. \\
&2.& &||a x|| > c, \mbox{ for all } 1 \leq a \leq M^2. \\
&3.& &c < 1/4. 
\end{aligned}
\end{equation}

The following definition will be useful in the next steps.
\begin{definition}
We shall call a fundamental interval $I(a_1,...,a_n)$ ``nice" if for all $1 \leq k \leq n$ one has that $||q_k x|| > c$, where $q_k$ is the $k$-th associated denominator of $I(a_1,...,a_n)$.
\end{definition}

{\bf Step 3.} {\em Properties of the sets $E_k$} \\
We construct the intervals $E_k$ by induction on $k$, with $E_0 = [0,1]$, and $E_k$ satisfying the following properties for $k \geq 1$:

\begin{enumerate}
\item Each interval in $E_k$ is a fundamental interval of order $k$ of the form $I(a_1,...,a_k)$, with with $a_i \leq 2M$, and $M^{k-1} \leq q_k < M^{k}$.
\item Each interval in $E_k$  is contained in a unique interval in $E_{k-1}$.
\item Each interval in $E_k$ is a ``nice" interval.
\item Every two intervals in $E_k$, contained in the same interval in $E_{k-1}$ are separated by a gap of size at least $\epsilon_k = M^{-(2k+3)}$.
\item If $y \in E_k$, then $||yq_n(x)|| > c$ for all $n \geq 0$, such that $c^{1/2}M^{2k} \leq q_n(x) < c^{1/2} M^{2(k + 1)}$.
\item Each interval in $E_{k-1}$ contains at least $m_k > M/32$ intervals from $E_k$.
\end{enumerate}

{\bf Step 4.} {\em Base case: $k = 1.$} \\
We first consider the fundamental intervals $I(a_1)$, with $1 \leq a_1 < M$ and call this collection $A_1$. The latter collection satisfies conditions 1 and 2. In addition, by our choice of $c$ in (\ref{c}) we have that  each of the $M-1$ considered intervals is ``nice", hence condition 3 holds. \\

The $M-1$ intervals considered decompose the interval $\big[\frac{1}{M}, 1\big]$ into $M-1$ intervals of the form $\big[\frac{1}{k + 1}, \frac{1}{k}\big]$ for $k = 1 ,\cdots M-1$. The length of the interval $\big[\frac{1}{k + 1}, \frac{1}{k}\big]$ is $\frac{1}{k(k+1)} > 1/M^2$, hence if we take every other of the $M-1$ intervals (i.e. the $1$-st, $3$-rd, $5$-th etc.), we will have at least $M/4$ intervals, separated by gaps of size at least $1/M^2$. Call the latter collection $B_1$ and then it is clear that every two intervals in $B_1$, are separated by a gap of size at least $\epsilon_1$, hence condition 4 holds.\\

From our choice of $c$ in (\ref{c}), we have that  $c^{1/2} M^{2(1 + 1)} < 1$, hence there are no $n$ such that $c^{1/2}M^{2} \leq q_n(x) < c^{1/2} M^{2(1 + 1)}$. This implies that $B_1$ satisfies condition 5 and since we showed that it contains at least $M/4$ intervals it satisfies condition 6. Setting $E_1:= B_1$, we establish the base case of our induction.\\

{\bf Step 5.} {\em Obtaining $E_{k+1}$ from $E_k$.} \\

Suppose that we have constructed $E_k$ and we wish to construct $E_{k+1}$. Take any interval in $I$ in $E_k$. We will construct a set of intervals $E_{k+1}^I$, contained in $I$, which satisfy conditions 1 through 6. Consequently $E_{k+1} = \cup_{I \in E_k} E_{k+1}^I$ will be our desired set. The main strategy is to start with a big set of candidate intervals for $E_{k+1}^I$, which we will refine to make sure that the remaining intervals satisfy the aforementioned conditions.\\

{\bf Step 6.} {\em The set $A_{k+1}^I$} \\
We aim to construct a set of intervals $A_{k+1}^I$, contained in $I$, satisfying conditions 1 and 2. By induction hypothesis we know that $I$ is of the form $I(a_1,...,a_k)$. Consider all fundamental intervals $I(a_1,...,a_k, s)$ and let the $k+1$-th order denominator associated with $I(a_1,...,a_k, s)$ be $q_{k+1,s}$. Let $A_{k+1}^I$ be the set of $I(a_1,...,a_k, s)$, which satisfy $s \leq 2M$, and $M^{k} \leq q_{k+1,s} < M^{k+1}$. We already know that $a_i < 2M$, by induction hypothesis, hence $A_{k+1}^I$ satisfies conditions 1 and 2.\\

We next wish to estimate the size of $A_{k+1}^I$. We have that $q_{k+1,s} = sq_{k} + q_{k-1}$. In particular, the denominators $q_{k+1,s}$ form an arithmetic progression, starting with $q_k + q_{k-1}$ and step $q_k$. From $q_{k-1} < M^{k-1} \leq q_{k} < M^{k}$, we know that $q_{k-1} + q_{k} < 2M^{k}$. So let $s_0$ be the smallest integer, such that  $M^{k} \leq s_0q_{k} + q_{k-1} < M^{k+1}$. From the inequalities for $q_k$, we know that $1 \leq s_0 \leq M$. In addition, we have by the minimality of $s_0$ that $(s_0 - 1)q_{k} + q_{k-1} < M^k$, which implies that $ s_0q_{k} + q_{k-1} < 2M^{k}$. The latter in particular implies that $ M^{k} \leq (s_0 + i)q_{k} + q_{k-1} < M^{k+1}$ for $i = 1,\cdots,[M/2]$. Thus we find that $A_{k+1}$ contains at least $M/2$ fundamental $k+1$-order intervals in $I$. \\ 

{\bf Step 7.} {\em The set $B_{k+1}^I$} \\

We denote by $B_{k+1}^I$ the subset of intervals in $A_{k+1}^I$ that are ``nice". Consequently the set $B_{k+1}^I$ will satisfy conditions $1,2,3$. Our main task in this step is to find a lower bound on the size of $B_{k+1}^I$.\\

Any interval in $A_{k+1}^I$ is of the form $I(a_1,...,a_k,s)$ and we already know that $q_i$ for $i \leq k$ satisfies $||q_kx|| > c$. So in particular to show that an interval in $A_{k+1}^I$ is ``nice", it is enough to show that $||q_{k+1,s}x|| >c$. We have that the numbers $q_{k+1,s}$ form an arithmetic progression starting from $q_{k+1,s_0}$ and step $q_k$. Since $||q_k x|| > c$, by induction hypothesis, we know from Lemma \ref{spec} that at least one of the following three inequalities holds for each $s$:

$$||sq_k + q_{k-1} x|| > c \hspace{5mm} ||(s + 1)q_k + q_{k-1} x|| > c \hspace{5mm} ||(s + 2)q_k + q_{k-1} x|| > c.$$

The latter implies that from each three consecutive intervals in $A^I_{k+1}$ at least one is ``nice". Since $A_{k+1}^I$ contains at least $M/2$ intervals, we conclude that among those at least $M/8$ are nice. Hence the size of $B_{k+1}^I$ is at least $M/8$.\\

{\bf Step 8.} {\em The set $C_{k+1}^I$} \\

We next wish to extract a set $C_{k+1}^I$ from $B_{k + 1}^I$, which satisfies condition 4. We have that each interval in $B_{k + 1}^I$  has size $\frac{1}{q_{k+1,s}(q_{k+1,s} + q_{k})}$, which by the inequalities for $q_{k+1,s}$ and $q_k$ is at least $\frac{1}{M^{k + 1}(M^{k+1} + M^k)} > \frac{1}{M^{2k+3}}$. Thus if we remove every other interval in $B_{k+1}^I$, we will be left with at least $M/16$, ``nice" intervals separated by distance at least $\epsilon_{k+1}$. Call the latter collection of intervals $C_{k+1}^I$.\\

{\bf Step 9.} {\em The set $E_{k+1}^I$} \\

In this final step we wish to extract a collection of intervals from $C_{k+1}^I$ that satisfies condition $5$. Firstly, for any number $q$, the set of points that satisfies $||qy|| \leq c$ is the set of intervals of the form $[\frac{p - c}{q}, \frac{p + c}{q}]$. These intervals are separated by distance $1/q$ and have length $2c/q$. In our case we wish to find intervals in $C_{k+1}^I$  that avoid all such intervals for $q = q_n(x)$, where $n$ satisfies $c^{1/2}M^{2k} \leq q_n(x) < c^{1/2} M^{2(k + 1)}$. 

\begin{claim} 
For each $q$ such that $c^{1/2}M^{2k} \leq q < c^{1/2} M^{2(k + 1)}$ we have that at most one interval $[\frac{p - c}{q}, \frac{p + c}{q}]$ intersects $I$.
\end{claim}
\begin{proof}
Suppose the contrary, then there are two intervals $I_1 = [\frac{p_1 - c}{q}, \frac{p_1 + c}{q}]$ and $I_2 = [\frac{p_2 - c}{q}, \frac{p_2 + c}{q}]$, which intersect $I$. Consequently there exist $a,b \in I$ such that $a \in I_1$ and $b \in I_2$.\\

We firstly have that $|p_1/q - p_2/q| \geq 1/q >  c^{-1/2} M^{-2(k + 1)}$ since $p_1 \ne p_2$ by assumption. On the other hand, by the triangle inequality we know that 
$$|p_1/q - p_2/q| \leq |p_1/q - a| + |b - a| + |p_2/q - b| \leq |I_1| + |I| + |I_2| = 4c/q + |I|.$$

We also have that $I = \big[\frac{p_k}{q_k}, \frac{p_k + p_{k-1}}{q_k + q_{k-1}} \big]$, hence $|I| = \frac{1}{q_k(q_k + q_{k-1})}$. Since $q_k > M^{k-1}$, we conclude that $|I| < 1/M^{2(k-1)}$. Combining this with the equations above and the inequalities for $q$ we get:

 $$c^{-1/2} M^{-2(k + 1)} < 1/q \leq |p_1/q - p_2/q| \leq 4c/q + |I| \leq 4c/c^{1/2}M^{2k} + 1/M^{2(k-1)} < 3/M^{2(k-1)}.$$

Consequently we get $3c^{1/2}M^4 > 1$, but by our choice for $c$ we have $1 > 3c^{1/2}M^4$, which is a contradiction. This concludes the proof of the claim.
\end{proof}

\begin{claim} 
For each $q$ such that $c^{1/2}M^{2k} \leq q < c^{1/2} M^{2(k + 1)}$ we have that the interval $[\frac{p - c}{q}, \frac{p + c}{q}]$ intersects at most one interval in $C_{k+1}^I$.
\end{claim}
\begin{proof}
Suppose the contrary, then there exist an iterval $J = [\frac{p - c}{q}, \frac{p + c}{q}]$ and two intervals $I_1, I_2 \in C_{k+1}^I$, such that $a \in J \cap I_1$ and $b \in J \cap I_2$.\\

Since $a,b \in J$, we know that $|a-b| \leq |J| = 2c/q \leq 2c^{1/2}M^{-2k}$. On the other hand, we know that $I_1$ and $I_2$ are separated by a distance at least $M^{-(2k + 3)}$, hence we obtain the inequalities:

$$2c^{1/2}M^{-2k} \geq |a-b| \geq M^{-(2k + 3)}.$$

The above implies that $2c^{1/2}M^3 \geq 1$, but from our choice of $c$ we have $2c^{1/2}M^3 < 3c^{1/2}M^4 \leq 1$, which is our desired contradiction.
\end{proof}

The above two claims show that for each $q_n(x)$ satisfying $c^{1/2}M^{2k} \leq q_n(x) < c^{1/2} M^{2(k + 1)}$, at most one interval $[\frac{p - c}{q_n(x)}, \frac{p + c}{q_n(x)}]$ intersects $I$ and if it does, then it intersects at most one interval in $C_{k+1}^I$. Let's denote by $S = \{ q_n(x) | c^{1/2}M^{2k} \leq q_n(x) < c^{1/2} M^{2(k + 1)}\}$ and by $T_{k+1}^I$ the set of intervals in $C_{k+1}^I$, which violate condition $5$. From the above discussion we have that $|T_{k+1}^I| \leq |S|$.

\begin{claim} 
$|S| \leq 2\log_{\lambda}(M) + 1$
\end{claim}
\begin{proof}
Let $|S| = s$ and $c^{1/2}M^{2k} \leq q_{n_1}(x) < q_{n_1}(x) < \cdots <  q_{n_s}(x) < c^{1/2} M^{2(k + 1)}$ be the elements of $S$ in ascending order. We have by the definition of $\lambda$ that $\frac{q_{n_{i + 1}}}{q_{n_{i}}} \geq \lambda$, for all $i = 1,...,s$, hence 

$$c^{1/2} M^{2(k + 1)} > q_{n_s}(x) \geq \lambda^{s-1}q_{n_1}(x) \geq \alpha^{s-1} c^{1/2}M^{2k}$$
Consequently $M^2 \geq \lambda^{s-1}$ or $2\log_{\lambda}(M) \geq s - 1$, which concludes the proof of the claim.
\end{proof}

From the above claim and the fact that $C_{k+1}^I$ contains at least $M/16$ intervals, we conclude that there is a subset of $C_{k+1}^I$ of size at least $M/16 - 2\log_{\lambda}(M) - 1$, satisfying condition $5$. Call the latter subset $E_{k+1}^I$. Then from our choice of $M$ we know that $M/16 - 2\log_{\lambda}(M) - 1 > M/32$, and so $E_{k+1}^I$ contains at least $M/32$ intervals, which shows $E_{k+1}^I$ satisfies conditions $1$ through $6$. We thus set $E_{k+1} = \cup_{I \in E_k} E_{k+1}^I$ and then we have our construction for the case $k +1$. The general result now follows by induction.\\

{\bf Step 10.} {\em $\cap_kE_{k} \subset A(x)$} \\

We now wish to show that $\cap_kE_{k} \subset A(x)$. Indeed, let $y \in \cap_kE_{k}$. We first remark that from condition $1$ each element of the continued fraction expansion of $y$ is bounded by $2M$, hence $y \in {\bf Bad}$. Since $x \in {\bf Bad}$ by assumption we know that $(x,y) \in {\bf Bad}^2$. Consequently there exist positive constants $c(x), c(y)$ such that  $n||ny|| > c(y)$ and $n||nx|| > c(x)$ for all $n \in \mathbb{N}$.

We know that each $q_n(x)$ satisfies $c^{1/2}M^{2k} \leq q_{n}(x) < c^{1/2} M^{2(k + 1)}$ for some $k$ (since $c^{1/2}M^{2} < 1$). Consequently $y \in E_{k}$ implies by condition $5$ that $||yq_n(x)|| > c$, so that $q_n(x)||xq_n(x)||||yq_n(x)|| > c \cdot c(x)$. In addition, for each $n$ we have that $y \in E_n$, hence $y$ belongs to a ``nice" interval of order $n$ and $||q_n(y)x|| > c$, so that $q_n(y)||xq_n(y)||||yq_n(y)|| > c \cdot c(y)$. And we see that if $r$ is a denominator of a convergent of $x$ or $y$, then $r||rx||||ry|| > c(x,y) = c\cdot \min\{c(x),c(y)\}$, i.e. $(x,y) \in A(x)$. Since $y \in \cap_k E_k$ was arbitrary we conclude that $\cap_k E_k \subset A(x)$.\\

{\bf Step 11.} {\em Lower bound on $\dim_H(\cap_k E_k)$} \\

We have from Lemma \ref{Fal46} that 

$$\dim_H(\cap_k E_k) \geq \liminf_{k \rightarrow \infty} \frac{\log(m_1 \cdots m_{k-1})}{-\log(m_k\epsilon_k)} $$
$$ = \liminf_{k \rightarrow \infty}\frac{\log((M/32)^{k-1})}{-\log((M/32)M^{-(2k+3)})} = \liminf_{k \rightarrow \infty}\frac{(k-1)\log(M/32)}{(2k+2)\log(M) + \log(32)} =  \frac{1}{2}\frac{\log(M/32)}{\log(M) }$$

{\bf Step 12.} {\em $\dim_HA(x) \geq 1/2$} \\

Since $\cap_k E_k \subset A(x)$ we have from the previous step that 

$$\dim_HA(x) \geq \frac{1}{2}\frac{\log(M/32)}{\log(M) },$$
where the latter holds for all sufficiently large $M$ (according to our choice). Taking $M \rightarrow \infty$, shows that $\dim_HA(x) \geq 1/2$, which concludes the proof of Proposition \ref{prop2}.

\section{Parallels with the Mixed Littlewood conjecture}
In this section we will translate Questions 1 and 2 from the introduction to the setting of the Mixed Littlewood conjecture (MLC). As will be seen the corresponding problem turns out to be simpler than the original one and one can find a straightforward solution, using the notion of Schmidt games. We recall the setting of MLC below.\\

Let $\mathcal{D} = \{d_k\}_{k\in \mathbb{N}}$ be a sequence of integers, greater than or equal to $2$. Set $t_0= 1$ and, for $n\geq 1$ let
$$t_n = \prod_{k = 1}^n d_k.$$
For $q \in \mathbb{N}$ we define 
$$\omega_{\mathcal{D}}(q) = \sup\{n \in \mathbb{N}: q\in t_n\mathbb{Z}\}$$
and
$$|q|_{\mathcal{D}} = \frac{1}{t_{\omega_{\mathcal{D}}(q)}} = \inf\{\frac{1}{t_n}: q \in t_n\mathbb{Z}\}.$$
When $\mathcal{D}$ is the constant sequence, equal to $p$, where $p$ is prime, then $|\cdot|_{\mathcal{D}}$ is the usual $p$-adic norm. Analogously to Littlewood's conjecture (LC) B. de Mathan and O. Teuli\'{e} proposed in \cite{deMathan04} the following mixed version of the conjecture\\

\setlength{\parindent}{0 pt}
{\bf Mixed Littlewood Conjecture:} For every real number $\alpha$
\begin{equation}\label{MLC}
\inf_{q \geq 1}q\cdot |q|_{\mathcal{D}}\cdot ||q\alpha|| = 0
\end{equation}

\setlength{\parindent}{20 pt}
Similarly to LC, MLC remains an open problem; however, much of the progress that has been made on LC, has managed to be adapted to MLC. For example, one can restrict the cases of interest in MLC to the case when $\alpha \in {\bf Bad}$ (otherwise $q\cdot |q|_{\mathcal{D}}\cdot ||q\alpha|| \leq q||q\alpha||$ would imply (\ref{MLC}) as the right hand side goes to $0$ upon taking infima, when $\alpha \notin {\bf Bad}$). Moreover, it was shown in \cite{Einsiedler} that when $d_k = p$ for all $k$, one has that the set of exceptions of MLC has Hausdorff dimension zero, similarly to the result in \cite{EKL06} for the classical case.\\
\setlength{\parindent}{20 pt}

In the above setting we now ask:
\setlength{\parindent}{0 pt}

 {\em Question 1'.} Does there exist $\alpha \in {\bf Bad}$ such that $\inf_{n \geq 1}q_n(\alpha)\cdot |q_n(\alpha)|_{\mathcal{D}}\cdot ||q_{n}(\alpha)\alpha|| > 0$ and $\inf_{n \geq 1}t_n\cdot |t_n|_{\mathcal{D}}\cdot ||t_n\alpha|| > 0$.\\

\vspace{2mm}

\setlength{\parindent}{0 pt}
 {\em Question 2'.} What is the size of $A \subset {\bf Bad}$,  such that for $\alpha \in A$ one has $\inf_{n \geq 1}q_n(\alpha)\cdot |q_n(\alpha)|_{\mathcal{D}}\cdot ||q_{n}(\alpha)\alpha|| > 0$ and $\inf_{n \geq 1}t_n\cdot |t_n|_{\mathcal{D}}\cdot ||t_n\alpha|| > 0$.\\

\vspace{2mm}

\setlength{\parindent}{20 pt}

Similarly to the questions in the introduction we restrict our attention to the case when $\alpha \in {\bf Bad}$, since that is the most relevant case for MLC. In addition, in view of the fact that the Lebesgue measure of {\bf Bad} is $0$, the set $A$ is also of measure zero, hence our discussion of the size of $A$, will be in terms of its Hausdorff dimension. The following proposition provides a complete answer to Questions 1' and 2' above.

\begin{proposition}\label{propMLC}
Let $\mathcal{D}$ and $\{t_n\}$ be as in the previous subsection. In addition, let 
$$A = \{ \alpha \in {\bf Bad} : \inf_{n \geq 1}q_n(\alpha)\cdot |q_n(\alpha)|_{\mathcal{D}}\cdot ||q_{n}(\alpha)\alpha|| > 0$$
$$ \mbox{ and }\inf_{n \geq 1}t_n\cdot |t_n|_{\mathcal{D}}\cdot ||t_n\alpha|| > 0 \}.$$ 
Then $A$ is $1/2$-winning and in particular $\dim_H A = 1.$
\end{proposition}
In the proof of Proposition \ref{propMLC} we will use the notion of Schmidt games, introduced by W. Schmidt in \cite{Schmidt66}, which we recall below.\\

The game is played by two players $A$ and $B$ for a given pair of numbers $\alpha, \beta \in (0,1)$, and goes as follows. First $B$ chooses a closed interval ${\bf B_0}$ such that $|{\bf B_0}| = \beta$ in $\mathbb{R}$ (here $|\cdot |$ denotes the length of an interval). After ${\bf B_n}$ has been chosen, $A$ chooses a closed interval ${\bf A_n} \subset {\bf B_n}$ such that $|{\bf A_n}| = \alpha |{\bf B_n}|$. Then $B$ chooses an interval ${\bf B_{n+1}} \subset {\bf A_n}$ such that $|{\bf B_{n+1}}|  = \beta |{\bf A_n}|$.  We obtain the nested sequence ${\bf B_0} \supset {\bf A_0} \supset {\bf B_1} \supset {\bf A_1} \cdots$. Notice that since the diameter of the intervals goes to zero (as $\alpha, \beta \in (0,1)$) and the intervals are closed then $\cap_{n=0}^\infty {\bf A_n} = \cap_{n=0}^\infty {\bf B_n}$ and the intersection is a single point. Now a set $Y \subset \mathbb{R}$ is called $(\alpha, \beta)$-winning if $A$ can play so that $\cap_{n=0}^\infty {\bf A_n} \in Y$, regardless of how $B$ plays. In addition, a set $Y$ is $\alpha$-winning if $Y$ is $(\alpha, \beta)$-winning for all $\beta \in (0,1)$. Schmidt showed in \cite{Schmidt66} that the intersection of countably many $\alpha$-winning sets is again $alpha$-winning, and in addition that an $\alpha$-winning set in the real line has Hausforff dimension one. \\

\begin{proof}(Proposition \ref{propMLC})
Let 
$$B = \{ \alpha \in \mathbb{R} : \exists c(\alpha) > 0 \mbox{ s.t. }||t_n\alpha|| > c(\alpha) \mbox{ for all }n\in \mathbb{N}  \}$$
and 
$$C  = \{ \alpha \in \mathbb{R} : \exists c(\alpha) > 0 \mbox{ s.t. } |q_n(\alpha)|_{\mathcal{D}} > c(\alpha) \mbox{ for all }n\in \mathbb{N}  \}.$$
Then it is easy to see that $A = B \cap C \cap {\bf Bad}$.\\

Since $\frac{t_n}{t_{n-1}} \geq 2$ for all $n \geq 2$ we have that $\{t_n\}$ is lacunary and so one can show, as mentioned by  Moshchevitin in \cite{Moshchevitin05}, that $B$ is $1/2$-winning. In addition, Schmidt showed in \cite{Schmidt66} that ${\bf Bad}$ is $1/2$-winning. Consequently our result will follow if we can show that $C$ is $1/2$-winning.\\

Let $\beta \in (0,1)$ be given. Set $R = \frac{2}{\beta} > 2$, and  $M = t_N$ where $N$ is sufficiently large so that $M >  R^{4k} + 2R^k $. In addition, define $V_s = \frac{1}{M}R^{(s+2)}$, $S_n = \{x \in \mathbb{R}: n\cdot||nx|| \leq 1\}$ and $T_s = \cup_{V_{s} \leq i < V_{s+1}}S_{iM}$. Notice that by our choice of $M$, $V_0 < 1$, and $V_s$ diverges. This implies that for each $i\geq 1$, $V_s \leq i< V_{s+1}$ for a unique $s$. In addition, we remark that the set $S_n$ is simply the disjoint union of all intervals of the form $[\frac{p}{n} - \frac{1}{n^2}, \frac{p}{n} + \frac{1}{n^2}]$.  \\

We now consider the following condition:\\

\setlength{\parindent}{0 pt}
(C1) ${\bf A_{2s}} \cap T_s = \emptyset$.
\setlength{\parindent}{0 pt}\\

Suppose that $A$ has a strategy that satisfies (C1) for every $s$. Then if $x \in \cap_{n = 0}^{\infty}{\bf A_{n}}$ one has that $x \notin T_s$ for all $s$. Consequently $x \notin S_{iM}$ for all $i \geq 1$, which implies that for each $i$ one has that $iM||iMx|| > 1$. The latter shows that $iM$ is not a denominator of a convergent for $x$. Since the latter is true for all $i$ we conclude that $M$ does not divide $q_n(x)$ for any $n \geq 1$. This implies that $|q_n(x)|_{\mathcal{D}} \geq \frac{1}{M}$ for all $n\in \mathbb{N}$ and hence $x \in C$. Thus if $A$ can play according to the above strategy we would obtain that $C$ is $(1/2, \beta)$-winning as desired.\\

\setlength{\parindent}{20 pt}
So let ${\bf A_0}$ be any subinterval of ${\bf B_0}$. Then $s = 0$ and by our choice of $M$ we know $\{i \in \mathbb{N}: V_0 \leq i < V_{1}\} = \emptyset$. Consequently $T_0$ is empty and ${\bf A_0}$ satisfies (C1). Next suppose that $s \geq 0$ and $A$ has chosen ${\bf A_{2s}}$, satisfying (C1). We wish to show that $A$ can play so that ${\bf A_{2(s+1)}}$ satisfies (C1).\\

\setlength{\parindent}{20 pt}
Let $\mathcal{E}_{s+1}$ be the collection of all intervals, intersecting ${\bf A_{2s}}$, in $S_{iM}$ for $i$ satisfying $V_{s+1} \leq i < V_{s+2}$. We make the following claim.\\

\setlength{\parindent}{0 pt}
{\bf Claim 1.} $A$ can play in such a way so that ${\bf A_{2s+1}}$, contains none of the intervals in $\mathcal{E}_{s+1}$.
\begin{proof} 
\setlength{\parindent}{20 pt}

If the collection $\mathcal{E}_{s+1}$ is empty there is nothing to prove. We thus may assume that $\mathcal{E}_{s+1}$ contains at least one interval $I$. The key idea in our approach is to show that all intervals in $\mathcal{E}_{s+1}$ have the same midpoint. \\

So suppose the latter is not true. Then one has two intervals, $I_1, I_2 \in \mathcal{E}_{s+1}$, given by $I_1 =  [\frac{p_1}{n} - \frac{1}{n^2}, \frac{p_1}{n} + \frac{1}{n^2}]$ and $I_1 =  [\frac{p_2}{m} - \frac{1}{m^2}, \frac{p_2}{m} + \frac{1}{m^2}]$, where $n = iM$ and $m = jM$ with $V_{s+1} \leq i,j < V_{s+2}$, and $\frac{p_1}{n} \ne \frac{p_2}{m}$. Then one has the inequality
$$| \frac{p_1}{n} - \frac{p_2}{m}| = \frac{|p_1m - p_2n|}{nm} \geq \frac{M}{nm}.$$
The last inequality follows from the fact that the numerator is non-zero, and divisible by $M$ as both $n$ and $m$ are.\\

Since $I_1, I_2 \in \mathcal{E}_{s+1}$, we have that there are $a \in I_1 \cap {\bf A_{2s}}$ and $b \in I_2 \cap {\bf A_{2s}}$. Since ${\bf A_{2s}}$ has length $\frac{1}{R^{2s}}$ we know that
$$|a - b| \leq \frac{1}{R^{2s}}.$$
Moreover, we have by the triangle inequality that
$$|a-b| + |a - \frac{p_1}{n}| +  |b - \frac{p_2}{m}|\geq |\frac{p_1}{n} - \frac{p_2}{m}| \geq \frac{M}{nm}$$
so that
$$|a-b| \geq \frac{M}{nm} - \frac{1}{m^2} - \frac{1}{n^2} \geq \frac{M}{M^2V_{s+2}^2} - \frac{2}{M^2V_{s+1}^2},$$
where the latter follows by $V_{s+2}> i,j \geq V_{s+1}$ and $n = Mi$, $m=Mj$. Combining the two inequalities above we get that 
$$\frac{1}{R^{2s}} \geq  \frac{M}{M^2V_{s+2}^2} - \frac{2}{M^2V_{s+1}^2} = \frac{M}{R^{2(s+4)}} - \frac{2}{2R^{(s+3)}} = \frac{M - 2R^2}{R^{2(s+4)}}.$$
The above implies 
$$(M - 2R^2) R^{2s} \leq R^{2(s+4)}  \hspace{2mm} \iff \hspace{2mm} (M - 2R^2)\leq R^{8},$$
which is a contradiction by our choice of $M > R^{8} + 2R^2$. This shows that indeed all intervals in $\mathcal{E}_{s+1}$ have the same midpoint.\\

So let us denote by $p/q$ the midpoint of an interval in $\mathcal{E}_{s+1}$. Let ${\bf B_{2s +1}} = [a,b]$. Then $A$ can choose an interval ${\bf A_{2s +1}} \subset {\bf B_{2s +1}}$, such that $p/q$ is not an interior point for ${\bf A_{2s +1}}$. The latter implies that ${\bf A_{2s +1}}$ contains no intervals from $\mathcal{E}_{s+1}$ and the claim is proved.
\end{proof}
Given Claim 1 above we know that ${\bf A_{2s+1}}$ and hence ${\bf B_{2s+2}}$ contain no intervals from $\mathcal{E}_{s+1}$. We now wish to show that $A$ can choose ${\bf A_{2s+2}}$ so that ${\bf A_{2s+2}}$ does not {\em intersect} any of the intervals in $\mathcal{E}_{s+1}$. To show the latter, suppose  ${\bf B_{2s+2}} = [a,b]$, then $A$ can choose $\big[ a+ \frac{b-a}{4}, b - \frac{b-a}{4}\big]$. We wish to show that this choice works.\\

So suppose that $I \in \mathcal{E}_{s+1}$ intersects ${\bf A_{2s+2}}$ at some point $t$. We know that $|I| = \frac{2}{n^2}$ for some  $MV_{s+1} \leq n < MV_{s+2}$. This shows that 
$$|I| \leq \frac{2}{R^{2(s + 3)}}.$$
Denote the right hand-side of the above by $\Delta$. Then we have, from $t \in I$, that $I \subset [t - \Delta, t+ \Delta]$. On the other hand, we have from $t \in {\bf A_{2s+2}}$ that $[t - \frac{b-a}{4}, t + \frac{b-a}{4}] \subset [a,b]$. Finally, notice that
$$\frac{|b-a|}{4} = \frac{1}{2}|{\bf A_{2s+2}}| = \frac{1}{2}\frac{1}{R^{2(s+1)}}.$$ 
We have 
$$\frac{1}{2}\frac{1}{R^{2(s+1)}} \geq \frac{2}{R^{2(s + 3)}} \hspace{2mm} \iff \hspace{2mm}  R^{4} > 4,$$
where the latter is true as $R > 2$. Consequently, $\frac{|b-a|}{4} \geq \Delta$, and thus $I \subset [t - \Delta, t+ \Delta] \subset [t - \frac{|b-a|}{4}, t + \frac{|b-a|}{4}] \subset [a,b] \subset {\bf A_{2s+1}}$. The latter is a contradiction, since by construction ${\bf A_{2s+1}}$ did not contain any of the intervals in $\mathcal{E}_{s+1}$. We thus conclude that ${\bf A_{2s+2}}$ does not intersect any of the intervals in $\mathcal{E}_{s+1}$, which shows $ {\bf A_{2s+2}} \cap T_{s+1} = \emptyset$. The result now follows by induction.\\

Since $\beta \in (0,1)$ was arbitrary, we conclude that $C$ is $1/2$-winning and the Proposition is proved.

\end{proof}

\end{document}